\documentclass[11pt]{amsart}
\usepackage[dvipdfmx]{graphicx, color}
\usepackage{amssymb, amsmath,amsthm,  url, verbatim, float}
\usepackage{epstopdf}
\usepackage{comment}
\usepackage{bm}
\newcommand{\Z}{{\mathbb{Z}}}
\newcommand{\NN}{{\mathbb{N}}}

\newcommand{\vol}{{\rm vol}}

\newcommand{\nc}{\newcommand}
\nc{\dmo}{\DeclareMathOperator}
\nc{\nt}{\newtheorem}

\nt{theorem}{Theorem}
\nt{lemma}{Lemma}

\newtheorem{thm}{{\bf Theorem}}
\newtheorem{lem}[thm]{{\bf Lemma}}
\newtheorem{cor}[thm]{{\bf Corollary}}

\newtheorem{rem}[thm]{Remark}

\newtheorem{ques}[thm]{Question}

\numberwithin{equation}{section}

\title[Volumes of fibered $2$-fold branched covers of $3$-manifolds]
{Volumes of  fibered $2$-fold branched covers of $3$-manifolds}

\date{\today}

\thanks{Hirose's research was partially supported by JSPS KAKENHI Grant Numbers JP16K05156 and JP20K03618. 
Kalfagianni's research was  partially
supported by NSF grants DMS-1708249 and DMS-2004155. 
Kin's research was partially supported by JSPS KAKENHI  Grant Numbers JP21K03247.}

\author{Susumu Hirose}
\address{%
Department of Mathematics,  
Faculty of Science and Technology, 
Tokyo University of Science, 
Noda, Chiba, 278-8510, Japan}
\email{%
hirose\_susumu@ma.noda.tus.ac.jp
}

\author{Efstratia  Kalfagianni}
\address{%
		Department of Mathematics, Michigan State University, 619 Red Cedar Road,
East Lansing, MI 48824, USA
}
\email{%
        kalfagia@msu.edu
}

\author{Eiko Kin}

\address{%
        Center for Education in Liberal Arts and Sciences, Osaka University, Toyonaka, Osaka 560-0043, Japan
}
\email{%
  kin.eiko.celas@osaka-u.ac.jp      
        }

\begin{document} 
\begin{abstract} 
We prove that for any closed, connected, oriented  $3$-manifold $M$, 
there exists an infinite family of $2$-fold branched covers of $M$ 
that are hyperbolic $3$-manifolds and surface bundles over the circle 
with arbitrarily large volume.
\end{abstract}

\maketitle

\section{Introduction}
\label{section_intro}

Sakuma \cite{sakuma} proved that every closed, connected, oriented $3$-manifold $M$ 
with a Heegaard splitting of genus $g$ admits a $2$-fold branched cover of $M$ 
that is a genus $g$ surface bundle over the circle $S^1$. 
See also Koda-Sakuma \cite[Theorem 9.1]{KS22}. 
Brooks  \cite{brooks} showed that the $2$-fold branched cover of $M$ in Sakuma's theorem can be chosen to be hyperbolic  
if $g \ge  \max(2, g(M))$, 
where $g(M)$ is the Heegaard genus of $M$.

Montesinos gave different proofs of results by 
Sakuma and Brooks by using  open book decompositions of $M$. 
To state his theorem, 
let $\Sigma = \Sigma_{g,m}$ be a compact, connected, oriented surface of genus $g$ 
with $m$ boundary components, 
and let $\Sigma_g= \Sigma_ {g,0}$. 
 The mapping class group $\mathrm{MCG}(\Sigma)$  is 
the group of isotopy classes of orientation-preserving 
self-homeomorphisms on $\Sigma$. 
  By the Nielsen-Thurston classification, 
elements  in $\mathrm{MCG}(\Sigma)$ fall into three types: 
periodic, reducible, pseudo-Anosov \cite{thurston:mappingtori}. 
For $f \in \mathrm{MCG}(\Sigma)$, we consider the mapping torus 
 $$T_{f}=\Sigma \times [-1,1]/_{(x,1)\sim ({ {f}(x),-1)}}.$$ 
 We call $\Sigma$ the fiber of $T_f$.  
 The $3$-manifold $T_f$ is a $\Sigma$-bundle over the circle $S^1$ with the monodromy $f$. 
 It is known by Thurston \cite{Thurston98} that 
 $T_f$ admits a hyperbolic structure of finite volume if and only if 
 $f$ is pseudo-Anosov. 
 The following result is a starting point of our paper.

\begin{thm}[Montesinos \cite{Mon1}]
\label{thm_Montesinos}
Let $M$ be a closed, connected, oriented $3$-manifold 
containing a hyperbolic fibered knot of genus $g_0 \ge 1$. 
Then there exists a 
$2$-fold branched cover of $M$ branched over  a $2$-component link 
that  is a hyperbolic $3$-manifold and a $\Sigma_{2g_0}$-bundle over $S^1$. 
\end{thm}

In this paper,  building on the approach of Montesinos, we prove the following result. 
Here $\mathrm{vol}(W)$ denotes the volume of a hyperbolic $3$-manifold $W$.

\begin{thm}
\label{thm_main}
Let $M$ be a closed, connected, oriented $3$-manifold 
containing a hyperbolic fibered knot of genus $g_0 \ge 2$. 
Then for any $g \ge g_0$ and $j \in \{1,2\}$,  
there exists an infinite family $\{N_{\ell}\}_{\ell \in \NN}$ of hyperbolic $3$-manifolds 
such that 
\begin{itemize}
\item [(a)]
$N_{\ell}$ is a $\Sigma_{2g+ j-1}$-surface bundle over $S^1$, 

\item[(b)]
$N_{\ell}$ is a $2$-fold branched cover of $M$ branched over a  $2j$-component  link, and 

\item  [(c)] 
the inequalities 
$$\frac{1}{2}\ g < \mathrm{vol}(N_{\ell})< \mathrm{vol}(N_{\ell+1})\ \hspace{2mm} \mbox{for} \hspace{2mm} \ell \in {\Bbb N}$$
hold. 
\end{itemize}
\end{thm}

By  Soma \cite{Soma84}, 
every closed  oriented, connected $3$-manifold $M$ contains a hyperbolic fibered knot of genus $g_0$ for some $g_0 \ge 1$. 
Equivalently 
there exists an open book decomposition $(\Sigma_{g_0,1}, h)$ of $M$, 
where the monodromy $h$ is isotopic to a pseudo-Anosov homeomorphism. 
By stabilizing open book decompositions along suitable arcs,  
one may assume that $M$ contains a hyperbolic fibered knot of genus $g$ for some $g \ge 2$, 
see Colin-Honda \cite{ColinHonda}, Detcherry-Kalfagianni \cite{DK:cosets} for example. 
Hence Theorem \ref{thm_main} applies to all  $3$-manifolds $M$.

Let $\mathcal{D}_g(M)$ be the subset of $\mathrm{MCG}(\Sigma_g)$ 
on the closed surface of genus $g$ 
consisting of elements $f$ such that 
its mapping torus $T_{f}$ is homeomorphic to a $2$-fold branched cover of $M$ branched over a link. 
The above result by Sakuma tells us that 
$\mathcal{D}_g(M) \ne \emptyset$ if $g \ge g(M)$, and 
there exist infinitely many pseudo-Anosov elements in $\mathcal{D}_g(M)$ if $g \ge \max(2, g(M))$, see \cite{brooks}. 
For a study of stretch factors of pseudo-Anosov elements of $\mathcal{D}_g(M)$, see \cite{HiroseKin20}. 
As an immediate corollary of Theorem \ref{thm_main}, we have following.

\begin{cor} 
Let $M$ be a closed, connected, oriented $3$-manifold 
containing a hyperbolic fibered knot of genus $g_0 \ge 2$.  
Then there exists an infinite family $\{\phi_g\}_{g=1}^{\infty}$ of pseudo-Anosov elements 
$\phi_g \in \mathcal{D}_{2g_0+ g-1}(M)$ 
such that the volume  $\mathrm{vol}(T_{\phi_g})$ of the mapping torus of $\phi_g$ 
goes to $\infty$ as $g \to \infty$. 
\end{cor}

We ask the following question. 

\begin{ques}
\label{question}

For $g$ sufficiently large, 
does the set $\mathcal{D}_g(M)$ contain an infinite family of pseudo-Anosov elements whose mapping tori have arbitrarily large volume? 
\end{ques}

Note that Futer, Purcell and Schleimer \cite{FPS22} give a positive answer to Question \ref{question} when $M= S^2 \times S^1$. 
Using the results of \cite{FPS22}, in Corollary \ref{positive}, we also obtain  a positive answer to Question \ref{question} when $M= S^3$ and $g$ is even.

\section{Fathi's theorem and volume variation}

This section is devoted to the proof of  a result 
which is a generalization of a theorem by Fathi \cite{Fathi}.  
Given a surface $\Sigma = \Sigma_{g,m}$ 
of genus $g$ with $m$ boundary components,  
let $\mathrm{MCG}(\Sigma)$ be the group of isotopy classes of orientation preserving self-homeomorphisms of $\Sigma$. 
In this section, we do not require that the maps and isotopies fix the boundary $\partial \Sigma$ of $\Sigma$ pointwise. 
In Sections \ref{section_curve-openbook} and \ref{section_proof-of-theorem}, 
we restrict our attention to the open book decompositions $(\Sigma, h)$ of  closed $3$-manifolds, 
where the monodromy $h: \Sigma \rightarrow \Sigma$ preserves $\partial \Sigma$ pointwise. 
By abuse of notations, we denote a representative of a mapping class $f \in \mathrm{MCG}(\Sigma)$ by the same notation $f$.

A simple closed curve $\gamma$ in $\Sigma$ is  \emph{essential} 
if it is not homotopic to a point or a boundary component. 
For simplicity, 
we may not distinguish between a simple closed curve $\gamma$ and its isotopy class $[\gamma]$. 
Let $\tau_{\gamma}$  denote the positive (i.e. right-handed)  Dehn twist about $\gamma$.

Let  $\gamma_1,\dots, \gamma_k$ be essential simple closed curves in $\Sigma$. 
We say that 
the set $\{\gamma_1,\dots, \gamma_k\}$  \emph{fills} $\Sigma$ if 
for each essential simple closed curve $\gamma'$ in $\Sigma$, 
there exists some $j \in \{1, \dots, k\}$ such that 
$i_{\Sigma}(\gamma', \gamma_j)>0$, 
where $i_{\Sigma}(\cdot,  \cdot)$ is the geometric intersection number on $\Sigma$. 
In this case, we also say that 
$\gamma_1,\dots, \gamma_k$  \emph{fill} $\Sigma$.

Given $f \in \mathrm{MCG}(\Sigma)$, 
we call $O_f(\gamma) =  \{f^{\ell}(\gamma) \ |\ \ell \in {\Bbb Z}\}$  the  \emph{orbit  of} $\gamma$ \emph{under} $f$. 
Strictly speaking, this is the set of isotopy classes $[f^{\ell}(\gamma)]$ of simple closed curves $f^{\ell}(\gamma)$. 
 We say that  orbits  of  $\gamma_1, \dots, \gamma_k$ under $f$ are \emph{distinct} 
if $O_f(\gamma_i) \ne O_f(\gamma_j)$  
for any $i,j \in \{1, \dots, k\}$ with $i \ne j$. 
Notice that 
$O_f(\gamma_i) =  O_f(\gamma_j)$ if and only if 
there exists  an integer $\ell \in {\Bbb Z}$ such that $f^{\ell}(\gamma_i) = \gamma_j$. 
We say that  the orbits of  $\gamma_1, \dots, \gamma_k$ under $f$ \emph{fill} $\Sigma$ if 
there exists an integer $n >0$ such that 
the set $\{f^{\ell}(\gamma_j)\ |\ j \in \{1, \dots, k\},\  \ell \in \{0, \pm 1, \dots, \pm n\}   \}$ fills $\Sigma$.

Suppose that  $\partial \Sigma \neq \emptyset$. 
A properly embedded arc $\alpha$ in $\Sigma$ is  \emph{essential} 
if it is not parallel to $\partial \Sigma$. 
As in the case of simple closed curves, 
we do not distinguish between an arc  $\alpha$ and its isotopy class $[\alpha]$. 
We allow that endpoints of the arcs are free to move around $\partial \Sigma$, 
and an arc $\alpha'$ that is isotopic to $\alpha$ may have the different endpoints from the ones of $\alpha$. 
Given $f \in \mathrm{MCG}(\Sigma)$, 
we  call $O_f(\alpha)=   \{f^{\ell}(\alpha) \ |\ \ell \in {\Bbb Z}\}$ the \emph{orbit of} $\alpha$ \emph{under} $f$.

Let $\alpha_1,\dots, \alpha_k$ be essential arcs. 
We say that  the orbits  of  $\alpha_1,\dots, \alpha_k$ under $f$ are   \emph{distinct} 
if $O_f(\alpha_i) \ne O_f(\alpha_j)$  
for any $i,j \in \{1, \dots, k\}$ with $i \ne j$.

 \begin{thm} 
 \label{hyperbolic}  
 Let $\gamma_1,\dots, \gamma_k$ be essential simple closed curves in $\Sigma = \Sigma_ {g,m}$, 
 where $k \ge 1$ and $3g-3+m>0$ (possibly $m=0$). 
 For any mapping class $f\in \mathrm{MCG}(\Sigma)$, 
 suppose that the orbits of $\gamma_1,\dots, \gamma_k$ under $f$ are distinct and fill $\Sigma$. 
 (i.e. $O_f(\gamma_i) \ne O_f(\gamma_j)$ for any $i,j \in \{1, \dots, k\}$ with $i \ne j$, and 
 the orbits of $\gamma_1,\dots, \gamma_k$ under $f$ fill $\Sigma$.)
 Then there exists $n \in {\Bbb N}$ which satisfies the following.
\begin{itemize}
\item [(a)] 
For any ${\bm n} =(n_1, \dots, n_k)\in {\mathbb Z}^k$  with $|n_i| \ge n$ for $i= 1, \dots, k$, 
the mapping class  
$$f_{\bm n} =\tau^{n_k}_{\gamma_k}  \dots  \tau^{n_1}_{\gamma_1} f \in \mathrm{MCG}(\Sigma)$$ 
is pseudo-Anosov.

\item [(b)] 
There exists a sequence  $\{{\bm n}_{\ell} \}_{\ell \in \NN} $ of the $k$-tuple of integers  
${\bm n}_{\ell} = (n_{\ell_1}, \dots, n_{\ell_k}) \in  {\mathbb Z}^k$ 
with $|n_{\ell_i}| \ge n$ for $i= 1, \dots, k$ 
such  that the mapping tori 
 $T_{f_{{\bm n}_{\ell}}}$ of $f_{{\bm n}_{\ell}}=\tau^{n_{\ell_k}}_{\gamma_k}  \dots  \tau^{n_{\ell_1}}_{\gamma_1} f $
are hyperbolic $3$-manifolds 
with  strictly increasing volumes: 
$$\frac{1}{2}\ k <  \vol(T_{f_{{\bm n}_{\ell}}})< \vol(T_{f_{{\bm n}_{\ell+1}}}) \hspace{2mm}\  \mbox{for} \hspace{2mm} \ 
\ell \in {\Bbb N}.$$
 \end{itemize}
   \end{thm}
    
In the case of $\partial \Sigma= \emptyset$, 
Theorem \ref{hyperbolic}(a) is due to Fathi \cite[Theorem 0.2]{Fathi}. 
The result by Fathi is a generalization of a theorem by Long-Morton \cite{LongMorton}. 
The argument we give below follows the line of the proof in \cite{LongMorton}. 
In our setting, the  mapping tori of the pseudo-Anosovs obtained in Theorem \ref{hyperbolic}(a) 
 are given by the Dehn filling along hyperbolic $3$-manifolds. 
 This allows us to use results on volume variation under the Dehn filling to prove  Theorem \ref{hyperbolic}(b).

\begin{proof}[Proof of Theorem \ref{hyperbolic}]
We take numbers  $0<t_1 < t_2< \dots  <t_k <1$. 
 Let  $\delta_i =\gamma_i \times \{ t_i \}$ be a curve lying on the fiber
 $F_i= \Sigma \times \{t_i \}$  for $i= 1, \dots, k$ of the mapping torus  $T_f$. 
 Now $L_k=\delta_1\cup \dots  \cup  \delta_k$ is a link in $T_{f}$. 
 \medskip

\noindent  
{\bf Claim 1.} 
Let $\mathcal{N}(L_k)$ denote a regular neighborhood of the link $L_k$.
Then the $3$-manifold $N =\overline{T_{f}\setminus \mathcal{N}(L_k)}$ is hyperbolic. 
 \medskip

 \noindent  
 Proof of Claim 1. 
 Since $\gamma_1,\dots, \gamma_k$ are  essential simple closed curves in $\Sigma$,  
 the $3$-manifold $N$ is irreducible and boundary irreducible. 
 
We first show that $N$ is atoroidal, i.e. $N$ contains no essential embedded tori.
 Assume that there exists  a torus $T$  embedded in $N$ that is incompressible and not peripheral. 
 Since the fundamental group of a thickened surface  $\Sigma \times I$, 
 where $I$ is an interval, 
 does not contain free abelian subgroups of rank $2$, 
the torus $T$ must intersect some of the fibers $F_i$ of $T_f$, where the curves $\delta_i$ lie.

Without loss of generality, we may suppose that $T$ intersects the fiber $F_1$, 
where $\delta_1 = \gamma_1 \times \{t_1\}$ lies. 
We identify $\Sigma$ with the $t_1$-level $F_1 = \Sigma \times \{t_1\}$ in $T_f$. 
Let $W_{\Sigma}$ denote the manifold obtained  
by cutting $T_f$ open along $\Sigma$. 
Since we assumed that  $0<t_1 < t_2< \dots  <t_k <1$, in the beginning of the proof,
and $\Sigma$ is identified with the $t_1$-level  $F_1=\Sigma \times \{t_1\}$,  
the level surface $\Sigma \times \{ 1\}=\Sigma \times \{ -1\}$ (as a set) is disjoint from $F_1$.
We can view $W_{\Sigma}$  as the identification space
\begin{equation}
\label{eq:WS}
W_{\Sigma} =(\Sigma \times [t_1, 1] \coprod \Sigma \times [-1,t_1])/_{(x,1)\sim ({ {f}(x),-1)}}, \end{equation}
and $T_f$ is obtained from 
 $W_{\Sigma}$ by identifying the two copies of $\Sigma \times \{ t_1\}$ in $\partial W_{\Sigma}$ by the identity map.

By using the irreducibility of $T_f$ and the incompressibility of $\Sigma$,  
we may isotope the torus $T$ so that 
all components of  $T \setminus \Sigma$ are annuli, and 
each component $A$ of $T \setminus \Sigma$ is either vertical with respect to the $I$-product 
 (i.e. $A$ runs around the $S^1$ factor of $T_f$), or 
 there exists an annulus $\widehat{A}$ in one copy of $\Sigma \subset \partial W_{\Sigma}$ 
 such that 
 $A \cup  \widehat{A}$ bounds a solid torus in $W_ {\Sigma}$ 
 and $\partial A$ lies in the same copy of $\Sigma$, 
 where the annulus $\widehat{A}$ sits. 
   The former and latter annuli are called the {\it vertical} and {\it horizontal} annuli respectively.  
 There are two types (A1), (A2) for a horizontal annulus $A$. 
\begin{enumerate}
\item[(A1)]
There exist no curves $\delta_i$ which is contained in the solid torus bounded by $A \cup  \widehat{A}$. 

\item[(A2)] 
There exists a curve $\delta_i$ which is contained in the solid torus bounded by $A \cup  \widehat{A}$. 

\end{enumerate}
If $A$ is a horizontal annulus of type (A2), then 
the curve $\delta_i$ in the condition of (A2) is unique: 
If $\delta_i$ and $\delta_j$ ($i \ne j$) are contained in the solid torus 
bounded by $A \cup  \widehat{A}$, 
then $\delta_i$ and $\delta_j$ are isotopic in $W_{\Sigma}$, 
which implies that 
$O_f(\gamma_i)= O_f(\gamma_j)$. 
This contradicts the assumption that the orbits of $\gamma_1, \dots, \gamma_k$ under $f$ are distinct. 
Hence the curve $\delta_i$ in  (A2) is unique. 
In particular $\partial A$ consists of two curves which are parallel to $\gamma_i \times \{t_1\}$ 
since $\Sigma$ is identified with the $t_1$-level $F_1 = \Sigma \times \{t_1\}$.

Notice that each horizontal annulus of type (A1) can be removed by an isotopy of the torus $T$, 
and hence 
we may suppose that 
each component of $ T \setminus \Sigma$ is a vertical annulus or 
a horizontal annulus of type (A2). 

If there exists a horizontal annulus of type (A2), then 
by replacing the fiber $F_i$ (containing the curve $\delta_i$) 
with $ F_1$ if necessary, we have a horizontal annulus 
$A_1$ of $ T \setminus \Sigma$ whose  components of $\partial A_1$ are parallel to $\delta_1 = \gamma_1 \times \{t_1\}$.

Suppose that there exist no vertical annuli of $T \setminus \Sigma$. 
Then a horizontal annulus $A_1$ can only connect to a horizontal annulus $A_2$ with $\partial A_2$ 
running parallel to $f^{\pm 1}(\gamma_1) \times \{t_1\}$. 
But then $T$ will be boundary parallel (peripheral)  in $N$, contrary to our assumption.

From the above discussion,  we may suppose that 
$T \setminus \Sigma$ contains  a vertical annulus  $A $.  
 Let $P$ denote the component of $\partial A$ on one copy of $\Sigma\times \{t_1\}$ on  $\partial W_{\Sigma}$.  
 By the construction of $W_{\Sigma}$ in \eqref{eq:WS}, the boundary $\partial A$  is disjoint from the level surface in
 $W_{\Sigma}$ resulting from the identification of
 $\Sigma \times \{ 1\}$ to $\Sigma \times \{ -1\}$ via ${(x,1)\sim ({ {f}(x),-1)}}$.
 The intersection of the annulus $A$ 
 with the later level surface is the curve resulting from
 $P\times \{1\}\sim f(P)\times \{-1\}$ under above identification of $\Sigma \times \{ 1\}$ to $\Sigma \times \{ -1\}$.
Thus the component of $\partial A$ on the second copy of $\Sigma \times  \{t_1\} \subset \partial W_{\Sigma}$ is $f(P)$.
That is $A$ runs from $P$ to $f(P)$ in $W_{\Sigma}$.

We have two cases. 
\begin{enumerate}
\item[(1)] 
$T \setminus \Sigma$ contains a horizontal annulus $A_1$, or 

\item[(2)] 
 all of the components of $T \setminus \Sigma$ are vertical annuli. 
\end{enumerate}
We first consider the case (1). As discussed earlier, without loss of generality, we may assume that $\partial A_1$ 
is formed by two curves parallel to $\gamma_1 \times \{t_1\}$. 

 Since the case that the horizontal annulus $A_1$ connects to another horizontal annulus was excluded earlier, we may now assume that
 $A_1$ connects a vertical annulus $A$. Now this vertical annulus $A$  
eventually connects to another horizontal annulus $A'$ of type (A2) 
such that the solid torus bounded by $A' \cup \widehat{A'}$ contains a curve $\delta_j$ 
for some $j \in \{1, \dots, k\}$. 

Assume that $j=1$. 
Then the torus $T$ must have a self-intersection in $N$, and this is a contradiction. 

Next, we assume that $j \in  \{2, \dots, k\}$. 
Then $\partial A'$ is formed by two curves parallel to $\gamma_j \times \{t_1\}$.
Recall that  the curves $P\times \{t_1\}$ and
$f(P)\times \{t_1\}$,  viewed on different copies of $\Sigma \times \{t_1\}\subset  \partial W_{\Sigma}$,
form the boundary of the vertical annulus $A$. This implies that on $\Sigma:=\Sigma \times \{t_1\}$ we have
$P, f(P) \in O_f(\gamma_j) \cap O_f(\gamma_1) \ne \emptyset$. 
This contradicts the assumption that
the orbits of  $\gamma_1, \dots, \gamma_k$ under $f$ are distinct.

We turn to the case (2). 
To form the torus $T$ from vertical annuli, 
we have $f^{m}(P)=P$ for some $m > 0$. 
Arguing as above,  we conclude that the curves $P, f(P), \dots, f^{m-1}(P) $ on $\Sigma \times \{t_1\}\subset T_f$  lie on the torus $T$. 
Since $T$ is embedded in $N$ and all of the components of $T \setminus \Sigma$ are vertical annuli, 
we have $\gamma_i\cap f^{j}(P)=\emptyset$ for any $i=1, \dots, k$ and $j=1, \dots, m$. 
Equivalently we have $f^{-j}(\gamma_i)\cap P=\emptyset$ for any $i=1,\dots,k$ and $j=1, \dots, m$.
For any $n\in  {\Bbb Z}$, write $n=m \ell -j$ for some $\ell \in  {\Bbb Z}$ and some $j=1, \dots, m$.
For  any $i=1, \dots, k$, we obtain 
$$f^{m \ell}( f^{-j}(\gamma_i)\cap P) = f^{m \ell}(f^{-j}(\gamma_i))\cap f^{m \ell }(P)= f^{m \ell -j}(\gamma_i)\cap P  
= f^n(\gamma_i)\cap P =  \emptyset.$$
Thus for any $i=1, \dots, k$,
the curve $P$ must be disjoint from  $O_f(\gamma_i)$. 
However this contradicts our assumption that the orbits of $\gamma_1, \dots, \gamma_k$ under $f$ fill $\Sigma$. 
This implies that $N$ is atoroidal.

To finish the proof of Claim 1, it is enough to show that 
$N$ contains no essential annuli. 
Suppose that there exists an  essential annulus $A$ in $N$.  
Then $N$ must be a Seifert manifold (see \cite[Lemma 1.16]{Hatcher}),  
and the components of $\partial N$ consist
of  fibers of the Seifert fibration of $N$. 
In $N$, we can find a copy of the fiber $\Sigma$ of $T_f$, say $S$, 
that  is disjoint from the components $T_1, \dots, T_k$ 
of  $\partial N $ 
that are created by drilling out the curves $\delta_1, \dots, \delta_k$.  
Then $S$ is a surface  that is essential in the Seifert manifold $N$ with non-empty boundary.
Since we assumed that $3g-3+m>0$, $S$ is not a torus or an annulus. 
Thus up to isotopy, we can make $S$ horizontal which means that
$S$ must intersect all the fibers of the Seifert fibration of $N$ 
transversely, see \cite[Proposition 1.11]{Hatcher}. 
Since $S$ is disjoint from the components $T_1, \dots, T_k$ of $\partial N$, 
 it cannot become horizontal. 
 This contradiction implies that $N$ contains no essential annuli. 
 Thus by work of Thurston \cite{thurston:survey}, the manifold $N$ is hyperbolic. 
This completes the proof of Claim 1. 
\medskip

We now prove the claim (a). 
We denote by $N_{\bm n}$, the mapping torus $T_{f_{\bm n}}$ of 
$f_{\bm n}= \tau^{n_k}_{\gamma_k}  \dots  \tau^{n_1}_{\gamma_1} f$ 
for ${\bm n}=(n_1, \dots, n_k) \in {\Z}^k$. 
We use the fact that   $N_{{\bm n}}$ 
 is obtained from $N$ by the Dehn filling, 
 where the  boundary component $T_i \subset \partial N$ corresponding to $\delta_i$ is filled. 
  Given ${\bm n}=(n_1, \dots, n_k)\in {\Z}^k$, 
 let $s_i$  denote the Dehn filling slope on  $T_i\subset \partial N$ to obtain $N_{\bm n}$ 
 for $i = 1, \dots, k$. 
 Since $N$ is hyperbolic, each torus boundary component of $N$ corresponds to a cusp of $T_f \setminus L_k$. 
 Taking a maximal disjoint horoball neighborhood about the cusps, each torus $T_i$ inherits a Euclidean structure, 
 well-defined up to similarity. The slope $s_i$ can then be given a geodesic representative. 
 We define the length of $s_i$, denoted by $\ell(s_i)$, 
 to be the length of this geodesic representative. 
 (Note that when $k > 1$, this definition of slope length depends on the choice of maximal horoball neighborhood. 
 See \cite{Purcell}.)

 The length  $\ell(s_i)$ of the slope $s_i$ is an increasing function of $|n_i|$. 
Let $\lambda >0$ denote the minimum length of the slopes, that is 
$$\lambda = {\min}\{\ell(s_i)\ |\ i = 1, \dots, k\}.$$
By Thurston's hyperbolic Dehn surgery theorem \cite{thurston:notes}, 
there exists $n\in \NN$ such that for all
 ${\bm n}=(n_1, \dots, n_k)\in {\Z}^k$ with $|n_i|>n$ for $i= 1, \dots, k$, 
 the resulting manifold $N_{\bm n} (= T_{f_{\bm n}})$ obtained by filling $N$ 
 is hyperbolic, and hence $f_{\bm n}$ is pseudo-Anosov.
 Thus the claim (a) holds.

 We turn to the claim (b). 
 As $|n_i|\to \infty$ for all  $i=1, \dots, k$, 
 the volumes of the filled manifolds $N_{\bm n}$'s  approach the volume of the $3$-manifold $T_f \setminus L_k$ from bellow.
 To make things more concrete, 
 we use an effective form  proved in \cite[Theorem 1.1]{fkp:filling}, 
 which states that if $\lambda >2\pi$, then  $N_{\bm n}$ is hyperbolic and we have
\begin{equation} 
\label{eq:variation}
 \left(1-\left(\frac{2\pi}{\lambda}\right)^2\right)^{3/2}\vol(T_{f}\setminus L_k)  \leqslant\ \vol(N_{\bm n})  <\  \vol(T_{f}\setminus L_k).
 \end{equation}
Since $T_{f}\setminus L_k$ is a hyperbolic $3$-manifold with at least $k$ cusps,  
we have 
$$k\ v_3 <  \vol(T_{f}\setminus L_k),$$ 
where $v_3= 1.01494\dots$ is the volume of the ideal regular tetrahedron, 
see  \cite[Theorem 7]{Adams}. 

On the other hand, by taking ${\bm n}=(n_1, \dots, n_k)\in {\Z}^k$ 
with all $|n_i|$ sufficiently larger than $n$, 
we can assure that 
$$ \frac{1}{2} < \left(1-\left(\frac{2\pi}{\lambda}\right)^2\right)^{3/2}.$$

By (\ref{eq:variation}), 
we obtain  

\begin{equation} 
\label{eq:variation2}
\frac{1}{2}\ k < \frac{1}{2}\ k v_3 < \frac{1}{2}\ \vol(T_{f}\setminus L_k)< \vol(N_{\bm n}).
\end{equation}

We set ${\bm n}_1 = {\bm n}$ with the above inequality $\frac{1}{2}\ k <\vol(N_{{\bm n}_1}) $.  
 Suppose that there exists a finite sequence  $\{{\bm n}_{\ell} \}_{\ell=1}^m$ 
 of  the  $k$-tuple of integers ${\bm n}_{\ell} \in {\Bbb Z}^k$ 
such that
$$\frac{1}{2}\ k< \vol(N_{ {\bm n}_1})< \dots <\vol(N_{ {\bm n}_m})< \vol(T_{f}\setminus L_k).$$
Now we  choose ${\bm n}_{m+1}=(n'_1, \dots, n'_k)\in {\Z}^k$ with
all $|n'_i|$  sufficiently larger than $n$ so that 
if we let $\lambda= \lambda_{{\bm n}_{m+1}}$ be the minimal length of the slopes corresponding to ${\bm n}_{m+1} \in {\Bbb Z}^k$, 
then we have 
$$\lambda >\frac{2\pi}{\sqrt {1-x_m^{2/3}}}>2\pi, \ \ {\rm where} \ \  
x_m =\frac{\vol(N_{ {\bm n}_m})}{\vol(T_{f}\setminus L_k)}.$$
Hence $\left(\frac{2\pi}{\lambda}\right)^2 < 1 - x_m^{2/3}$. 
This tells us that 
$$\frac{\vol(N_{ {\bm n}_m})}{\vol(T_{f}\setminus L_k)} = x_m
<  \left(1-\left(\frac{2\pi}{\lambda}\right)^2\right)^{3/2}.$$
Thus 
$$\vol(N_{ {\bm n}_{m}}) < \left(1-\left(\frac{2\pi}{\lambda}\right)^2\right)^{3/2}\vol(T_{f}\setminus L_k).$$ 
By \eqref{eq:variation}, we have 
$$\left(1-\left(\frac{2\pi}{\lambda}\right)^2\right)^{3/2}\vol(T_{f}\setminus L_k)  \le \vol(N_{ {\bm n}_{m+1}}).$$
Putting them together, we obtain 
\begin{equation*} 
\vol(N_{ {\bm n}_{m}}) < \left(1-\left(\frac{2\pi}{\lambda}\right)^2\right)^{3/2}\vol(T_{f}\setminus L_k)  \le \vol(N_{ {\bm n}_{m+1}}), 
\end{equation*}
and the conclusion follows inductively. 
This completes the proof of Theorem \ref{hyperbolic}. 
 \end{proof} 
  
  \section{Curves on surfaces and open book decompositions}
  \label{section_curve-openbook}
  
 In this section, 
 we quickly review curve graphs and open book decompositions of $3$-manifolds. 
 We prove a lemma 
 that is needed for the proof of Theorem \ref{thm_main}.
 
\subsection{Curves on surfaces} 
\label{subsec:curves}

Suppose that $g \ge 2$. 
The \textit{curve graph} ${\mathcal C}(\Sigma)$ for $\Sigma=\Sigma_{g,m}$  is defined as follows. 
  The set of vertices ${\mathcal C}_0(\Sigma)$ is the set of isotopy classes of  essential simple closed curves. 
 Two vertices  in $ {\mathcal C}_0(\Sigma)$ are connected by an edge 
 if they can be represented by  disjoint  essential simple closed curves.  
 The space  ${\mathcal C}(\Sigma)$  is a geodesic metric space with the path metric $d(\cdot, \cdot)$ 
 that assigns length 1 to each edge of the graph. 
The mapping class group  $\mathrm{MCG}(\Sigma)$ acts on ${\mathcal C}(\Sigma)$  as isometries. 

\begin{center}
\begin{figure}[t]
\includegraphics[height=2.3cm]{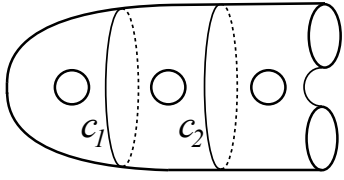}
\caption{Case $(g,m)= (3.2)$. An example of simple closed curves $c_1, \dots, c_{g-1}$ in $\Sigma_{g,m}$. }
\label{fig_arc}
\end{figure}
\end{center}

 \begin{lem}
 \label{lem_curve} 
Let $f \in \mathrm{MCG}(\Sigma)$ be a pseudo-Anosov mapping class 
defined on $ \Sigma= \Sigma_{g,m}$, where $g \ge 2$.  
Then for any $1 \le k \le g$, there exist mutually disjoint, essential simple closed curves 
$\gamma_1, \dots, \gamma_k$ in $\Sigma$ such that 
\begin{enumerate}
\item[(a)] 
the orbits of $\gamma_1, \dots, \gamma_k$   under $f$ are distinct and  fill $\Sigma$, and 

\item[(b)] 
the surface $\Sigma \setminus \{\gamma_1, \dots, \gamma_k\}$ obtained from $\Sigma$ cutting along $\gamma_1 \cup \dots \cup \gamma_k$ 
is connected. 

\end{enumerate}
\end{lem}

\begin{rem}
\label{rem_discussion}
The curve graph $\mathcal{C}(\Sigma)$ is locally infinite, 
i.e.  for each vertex $v \in \mathcal{C}_0(\Sigma)$, 
there exist infinitely many vertices of $\mathcal{C}_0(\Sigma)$ that are at distance $1$ from $v$.
It is not hard to see that 
 if $d (a, b) \ge 3$ for  $a, b \in \mathcal{C}_0(\Sigma)$, then 
 $\{a, b\}$ fills $\Sigma$. 
Furthermore, if $f \in \mathrm{MCG}(\Sigma)$ is pseudo-Anosov, 
then 
the distance $d(a, f^n(a))$ grows linearly with $|n|$ for any $a \in \mathcal{C}_0(\Sigma)$, 
see \cite[Proposition 4.6]{MaMi}. 
 We define the ball 
 $$B_1(a) = \{b   \in \mathcal{C}_0(\Sigma)\ |\ d(a, b) \le 1\}.$$
 Since $d (a, f^n(a)) \to \infty$ as $|n| \to \infty$, 
 a single orbit $O_f(a)$  of $a \in  {\mathcal C}_0(\Sigma)$ under $f$ fills $\Sigma$. 
 Take any $b \in \mathcal{C}_0(\Sigma)$. 
 Then the cardinality of the set $B_1(a) \cap O_f(b)$ is finite, 
 since $d (b, f^n(b)) \to \infty$ as $|n| \to \infty$. 
 Moreover $B_1(a) \setminus O_f(b)$ is an infinite set, 
 since  $\mathcal{C}(\Sigma)$ is locally infinite. 
 Hence one can pick an element of $B_1(a) \setminus O_f(b)$ 
 at distance $1$ from $a$. 
 \end{rem}

\begin{proof}[Proof of Lemma \ref{lem_curve}]
We first  take mutually disjoint, essential simple closed curves 
$c_1, \dots, c_{g-1}$ in $\Sigma$ so that 
the surface obtained from $\Sigma$ by cutting along $c_1 \cup \dots \cup c_{g-1}$ 
has $g$ connected components $\Sigma^{(1)},\dots,\Sigma^{(g)}$, 
each of which is a surface of genus $1$ with nonempty boundary. 
See Figure \ref{fig_arc}. 

For each $1 \le k \le g$, 
there exists an infinite family $\{a_i^{(k)} \}_{i \in {\Bbb N}}$ of $ \mathcal{C}_0(\Sigma)$ such that 
$a_i^{(k)} \ne a_j^{(k)} \in \mathcal{C}_0(\Sigma)$ if $i \ne j$ and 
$a_i^{(k)}$ is represented by a non-separating simple closed curve in the surface $\Sigma^{(k)}$. 
Then 
\begin{equation}
\label{equation_distance-one}
d(a_i^{(k)}, a_j^{(\ell)})= 1 
\hspace{2mm} \mathrm{if}\hspace{2mm} 
  k \ne \ell \hspace{2mm}\mbox{and} \hspace{2mm}  i,j \in {\Bbb N}. 
\end{equation}
In the family $\{a_i^{(1)}\}_{i \in {\Bbb N}}$, 
take any $a_{i_1}^{(1)} = [\gamma_1] $. 
Then 
the orbit of $\gamma_1$ under $f$ fills $\Sigma$ by Remark  \ref{rem_discussion}. 
The statement of the lemma holds in the case $k=1$. 

We turn to the case $k=2$. 
By (\ref{equation_distance-one}), we have 
$\{a_i^{(2)}\}_{i \in {\Bbb N}}  \subset B_1([\gamma_1]) = B_1(a_{i_1}^{(1)})$. 
By Remark \ref{rem_discussion}, one sees that 
$\{a_i^{(2)}\}_{i \in {\Bbb N}}  \cap O_f(\gamma_1)\  \Bigl(\subset B_1([\gamma_1]) \cap O_f(\gamma_1) \Bigr)$ is finite. 
Hence one can pick an element 
$a_{i_2}^{(2)} = [\gamma_2] \in \{a_i^{(2)}\}_{i \in {\Bbb N}}    \setminus O_f(\gamma_1)$. 
Then the orbits of $\gamma_1$ and $\gamma_2$ under $f$ are distinct by the choice of $\gamma_2$. 
By (\ref{equation_distance-one}), two curves $\gamma_1 $ and $\gamma_2 $ are disjoint. 
Moreover the orbits of $\gamma_1$ and $ \gamma_2$ under $f$ fill $\Sigma$, 
since a single orbit of $\gamma_1$ under $f$ fills $\Sigma$. 
Since $\gamma_1$ (resp. $\gamma_2$) is non-separating in the surface $\Sigma^{(1)}$ (resp. $\Sigma^{(2)}$), 
one sees that $\Sigma \setminus \{\gamma_1, \gamma_2\}$ is connected. 
Thus the statement of the lemma holds in the case $k=2$.

Similarly for $3 \le k \le g$, 
one can find the vertices 
$a_{i_3}^{(3)} = [\gamma_3], \dots, a_{i_k}^{(k)} = [\gamma_k]  $ 
such that 
the orbits of $\gamma_1, \gamma_2$, $\gamma_3, \dots, \gamma_k$ 
under $f$ are distinct and fill $\Sigma$. 
By (\ref{equation_distance-one}), 
$\gamma_1, \gamma_2, \gamma_3  \dots, \gamma_k$ are mutually disjoint. 
Each $\gamma_i$ is non-separating in the surface $\Sigma^{(i)}$ for $i = 1, \dots, k$, 
and this implies that 
$\Sigma \setminus \{\gamma_1, \dots, \gamma_k \}$ is connected. 
This completes the proof.
\end{proof}

 \subsection{Open book decompositions of closed $3$-manifolds}  
 \label{subsection_openbook}

An {\it open book decomposition} of $M$ is a pair $(K, \theta)$, 
where $K$ is a link in $M$ and $\theta: M \setminus K \rightarrow S^1$ is a fibration 
whose fiber is an interior of a Seifert surface of $K$. 
We call $K$  the {\it binding} of the open book decomposition. 
We also call $K$  the {\it fibered link} in $M$. 
An open book decomposition of $M$ is determined by 
the closure $\Sigma= \overline{\theta^{-1}(t)}  \subset M$ of a fiber $\theta^{-1}(t)$ ($t \in S^1$) 
of the fibration $\theta$
together with the monodromy $h : \Sigma \rightarrow \Sigma$ with $h|_{\partial \Sigma}= \mathrm{id}$. 
Conversely, 
each pair  $(\Sigma, h)$ with $h|_{\partial \Sigma}= \mathrm{id}$ 
gives rise to an open book decomposition of some $3$-manifold $M$ 
as the {\it relative mapping torus} of  $h$, i.e. 
$M$ is homeomorphic to the quotient of the mapping torus $T_h$ of $h$ 
under the identification 
$(y,t)\sim (y, t')$ for all  $y \in\partial \Sigma$ and $t,t'\in [-1, 1]$. 
We also call such a pair $(\Sigma, h)$ the  open book decomposition of a $3$-manifold.

 By the proof of  \cite[Theorem 1.1]{ColinHonda} by Colin-Honda, the following result holds. 
See also  Detcherry-Kalfagianni \cite[Propositions 4.9, 4.10]{DK:cosets}. 

  \begin{thm}
\label{thm:colin-honda} 
Let $M$ be a closed, connected, oriented $3$-manifold containing a hyperbolic fibered knot of genus $g_0 \ge 2$. 
Then for any $g \ge g_0$ and $j \in \{1,2\}$, 
the manifold $M$ admits an open book decomposition 
$(\Sigma_{g,j}, h_{g,j})$, 
 where $\partial \Sigma_{g,j}$ has $j$ components and 
 $h_{g,j}$ is isotopic to a pseudo-Anosov homeomorphism. 
\end{thm}

\section{Proof of Theorem \ref{thm_main}}
\label{section_proof-of-theorem}

Theorem \ref{thm_main} in Section \ref{section_intro} follows from the following result. 

\begin{thm}
\label{adding_arc}
Let $M$ be a closed, connected, oriented $3$-manifold containing a hyperbolic fibered knot of genus $g_0 \ge 2$. 
Then for any $g \ge g_0$, $j \in \{1,2\}$ and $2 \le k \le g$,   
there exists  $n  \in {\Bbb N}$ which satisfies the following. 
For any 
 ${\bm n}=(n_1, \dots, n_k) \in {\mathbb Z}^k$  with $|n_i| \ge n$ for $i= 1, \dots, k$, 
there exists a hyperbolic $3$-manifold $N_{\bm n}$ such that
\begin{itemize}
\item [(a)] $N_{\bm n}$ is a $\Sigma_{2g+ j-1}$-surface bundle over $S^1$, 
\item [(b)] $N_{\bm n}$ is a $2$-fold branched cover of $M$ 
branched over a $2j$-component  link, and  
\item [(c)] 
there exists a sequence $\{{\bm n}_{\ell} \}_{\ell \in \NN} $ of the $k$-tuple of integers  
${\bm n}_{\ell} = (n_{\ell_1}, \dots, n_{\ell_k}) \in  {\mathbb Z}^k$ 
with $|n_{\ell_i}| \ge n$ for $i= 1, \dots, k$ 
such that 
$$\frac{1}{2}\ k < \vol(N_{ {\bm n}_{\ell}})< \vol(N_{ {\bm n}_{\ell+1}}) \hspace{2mm}\   \mbox{for}\hspace{2mm} \ \ell \in {\Bbb N}.$$ 

\end{itemize}
\end{thm}

\begin{proof} 
By Theorem \ref{thm:colin-honda},  for any $g \ge g_0$ and $j \in \{1,2\}$,  
there exists an open book decomposition 
$(\Sigma_{g,j}, h_{g,j})$ of $M$, 
 where 
 $h_{g,j}$ is isotopic to a pseudo-Anosov homeomorphism. 
We set $F_ {g,j}= \Sigma_ {g,j}$. 
Then by Lemma \ref{lem_curve}, 
we have 
mutually disjoint, essential simple closed curves $\gamma_1,\dots, \gamma_k$ in $F_{g, j}$ 
such that the orbits of  $\gamma_1,\dots, \gamma_k$ under $h_{g, j}$ are distinct and fill $F_{g, j}$. 
Moreover $F_{g,j} \setminus \{\gamma_1, \dots, \gamma_k\}$ is connected.

Let $B= \partial F_{g, 1}$  when $j=1$, and 
let $B$ and $ B'$ be the components of $\partial F_{g, 2}$ when $j=2$. 
When $j=1$, let $\beta_1, \dots, \beta_k$ be properly embedded, mutually disjoint arcs in 
$F_{g,1} \setminus \{\gamma_1, \dots,  \gamma_k\}$ 
so that one of the endpoints of each $\beta_i$ lies on $\gamma_i$ and 
the other endpoint of $\beta_i$ lies on $B= \partial F_{g, 1}$. 
Since $F_{g,1} \setminus \{\gamma_1, \dots , \gamma_k\}$ is connected, 
one can choose those arcs 
$\beta_1, \dots, \beta_k \subset F_{g,1} \setminus \{\gamma_1, \dots, \gamma_k\}$ so that they are mutually disjoint. 
When $j=2$,  let $\beta_1, \dots,  \beta_k$ be properly embedded, mutually disjoint arcs in 
$F_{g,2} \setminus \{\gamma_1, \dots, \gamma_k\}$ 
so that one of the endpoints of each $\beta_i$ lies on $\gamma_i$ and 
the other endpoint of $\beta_i$ lies on $B$ (resp. $B'$) if $i = 1, \dots, k-1$ (resp. $i = k$).
In both cases $j= 1,2$, 
consider a small neighborhood $\mathcal{N}= \mathcal{N}(\gamma_i \cup \beta_i)$ in $F_{g,j} $. 
We set $\alpha_i $ to be a component  of $\partial \mathcal{N}  \setminus \partial F_{g,j}$ 
which is not parallel to $\gamma_i$. 
See Figures \ref{fig_fg1}(1), \ref{fig_fg2}(1). 
Then $\alpha_1, \dots, \alpha_k$ are mutually disjoint, essential arcs  in $F_{g,j}$.

 We claim that 
 the orbits of $\alpha_1, \dots, \alpha_k$ under $h_{g,j}$ are distinct. 
 Assume that 
 $O_{h_{g,j}}(\alpha_i) = O_{h_{g,j}}(\alpha_{i'})$ for some $i, i' \in \{1, \dots, k\}$ with $i \ne i'$.  
 This implies that 
 $O_{h_{g,j}}(\gamma_i) = O_{h_{g,j}}(\gamma_{i'})$, 
 since $\gamma_i$ is obtained from each $\alpha_i$ by concatenating with an arc of $B$ or $B'$.   
 This contradicts the choice of $\gamma_1, \dots, \gamma_k$.

Let us consider the closed surface $\Sigma_{2g+j-1}=DF_{g, j}$ of genus $2g+j-1$ 
that is obtained as the double $DF_{g, j}$ of $F_{g, j}$ along  $\partial F_ {g, j}$. 
There exists an involution 
$$\iota : \Sigma_{2g+j-1}\rightarrow \Sigma_{2g+j-1}$$ 
that interchanges the two copies of $F_{g, j}$ and $\iota|_ {\partial F_{g,j}}= \mathrm{id}$ holds. 
(Notice that $\iota$ is orientation reversing.)  
For the above essential arc $\alpha_i$, there
is a corresponding arc $\iota(\alpha_i)$ on the second copy of $F_{g,j}$  
so that $\widehat{\gamma}_i= \alpha_i \cup \iota(\alpha_i)$ becomes 
an essential simple closed curve in $\Sigma_{2g+j-1}$. 
Since $\alpha_1, \dots, \alpha_k$ are mutually disjoint, 
$\widehat{\gamma}_1, \dots, \widehat{\gamma}_k$ are mutually disjoint, essential simple closed curves in $\Sigma_{2g+j-1}$. 
See Figures \ref{fig_fg1}(2), \ref{fig_fg2}(2).

Let 
$$ \widehat{h}_{g,j} = h_{g,j}\# h_{g,j}^{-1} : \Sigma_{2g+j-1} \rightarrow  \Sigma_{2g+j-1} $$ 
be a homeomorphim induced by $h_{g,j}$. 
More precisely, 
$\widehat{h}_{g,j}(x) = h_{g,j}(x)$ if $x$ is in one copy of $F_{g,j}$ and 
$\widehat{h}_{g,j}(\iota (x)) = \iota (h^{-1}_{g,j}(x))$ if $\iota (x)$ is in the second copy of $F_{g,j}$. 

Note that  $\widehat{h}_{g,j}$ is a reducible homeomorphism,  
since $\widehat{h}_{g,j}$ preserves the essential simple closed curves $\partial F_{g, j} \subset \Sigma_{2g+j-1}$.

 \begin{center}
\begin{figure}[h]
\includegraphics[height=3cm]{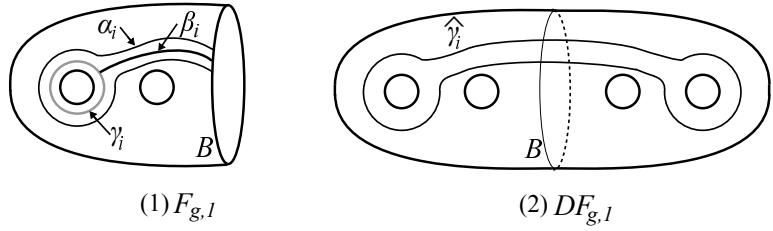}
\caption{Case $(g,j)= (2,1)$. 
(1) The arc $\alpha_i$ in $F_{g,1} (= \Sigma_{g,1})$. 
(2) The simple closed curve $\widehat{\gamma}_i$ in $\Sigma_{2g}= DF_{g,1}$.}
\label{fig_fg1}
\end{figure}
\end{center}

 \begin{center}
\begin{figure}[h]
\includegraphics[height=3cm]{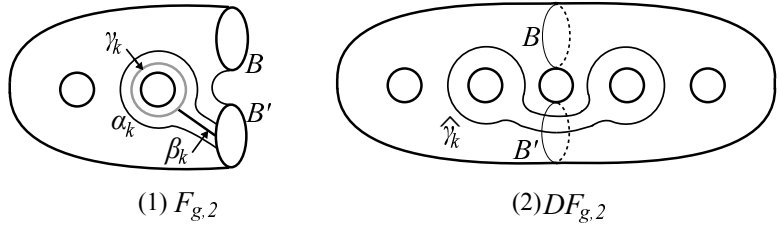}
\caption{
Case $(g,j)= (2,2)$. 
(1) The arc $\alpha_k$ in $F_ {g,2} (= \Sigma_{g,2})$. 
(2) The simple closed curve $\widehat{\gamma}_k$ in $\Sigma_{2g+1}= DF_{g,2}$.}
\label{fig_fg2}
\end{figure}
\end{center}

\noindent
{\bf Claim 1.}  
For $j \in \{1,2\}$, 
the orbits of $\widehat{\gamma}_1,\dots, \widehat{\gamma}_k$ 
under $\widehat{h}_{g,j}$  are distinct. 
\medskip
\\
Proof of Claim 1. 
By the definition of $\widehat{h}_{g,j}$, 
we have $\widehat{h}_{g,j}|_ {\partial F_{g,j}}= \mathrm{id}$, and 
the orbits of the arc $\alpha_i \subset \widehat{\gamma}_i$  
under $\widehat{h}_{g,j}$ are contained in one copy of $F_{g,j}$ 
for $i= 1, \dots, k$. 
Claim 1 follows, since the orbits of $\alpha_1, \dots, \alpha_k$ under $h_{g,j}$ are distinct.  
\medskip

\noindent
{\bf Claim 2.}  
For $j \in \{1,2\}$, 
the orbits of $\widehat{\gamma}_1,\dots, \widehat{\gamma}_g$ under $\widehat{h}_{g,j}$ fill $\Sigma_{2g+ j-1}$. 
\medskip
\\
Proof of Claim 2.  
We prove the claim when $j=2$. 
(The proof for the case of $j=1$ is similar.) 
By the proof of Lemma \ref{lem_curve}, 
a single orbit of $\gamma_1$  under $h_{g,2}$ fill $F_{g, 2}$. 
This means that there exists an integer $n>0$ such that 
each component of 
$ F_{g, 2} \setminus \{h_{g, 2}^{\ell} (\gamma_1)\ |\  \ell \in \{0, \pm 1, \dots, \pm n\} \}$ 
 is a disk or a once-holed disk. 
 Let $A$ (resp. $A'$) 
 be the annular component of 
 $ F_{g, 2} \setminus \{h_{g, 2}^{\ell} (\gamma_1)\ |\ \ell \in \{0, \pm 1, \dots, \pm n\} \}$ 
 such that one of the boundary components of $A$ (resp. $A'$) coincides with 
 $B \subset \partial F_{g,2}$ (resp. $B' \subset \partial F_{g,2}$). 
 Then the annulus $A$ 
 is cut into disks  by cutting $F_{g,2}$ along the arc $\alpha_1$ 
 (since $\partial \alpha_1$ lies on $B$).  
 The other annulus $A'$ 
 is also cut into disks by cutting $F_{g,2}$ along the arc $\alpha_k$ 
 (since $\partial \alpha_k$ lies on $B'$).  
 Thus 
 the surface obtained from the double $\Sigma_{2g+1}= DF_{g,2}$ 
 by cutting along all $\widehat{h}_{g,2}^{\ell} (\widehat{\gamma}_1)$ and 
 $\widehat{h}_{g,2}^{\ell} (\widehat{\gamma}_k)$ ($\ell \in \{0, \pm 1, \dots, \pm n\}$)  
 is a disjoint union of disks. 
 This means that the orbits of $ \widehat{\gamma}_1,  \widehat{\gamma}_k$ under $\widehat{h}_{g,2}$  fill $\Sigma_{2g+1}$. 
 Thus the orbits of $ \widehat{\gamma}_1, \dots,  \widehat{\gamma}_k$ under $\widehat{h}_{g,2}$  fill $\Sigma_{2g+1}$. 
 This completes the proof of Claim 2. 
 \medskip

We build the mapping torus 
$$T_{\widehat{h}_{g,j}}= \Sigma_{2g+j-1}\times [-1, 1] /_{(x,1)\sim ({{\widehat{h}_{g,j}}(x),-1)}}.$$ 

\noindent
{\bf Claim 3.} 
For $j \in \{1,2\}$, 
the mapping torus  $T_{\widehat{h}_{g,j}}$ is a $2$-fold branched cover of  $M$
 branched over a $2j$-component link. 
 \medskip
 
 \noindent
 Proof of Claim 3. 
 The statement of Claim 3 follows from the proof of \cite[Lemma 1]{Mon1}. 
Here we prove the claim for completeness. 
Consider the involution 
$u:  T_{\widehat{h}_{g,j}}\rightarrow  T_{\widehat{h}_{g,j}}$ defined by 
 $$u(x, t)=(\iota(x), -t) \ { \rm{for }} \  (x,t)\in \Sigma_{2g+j-1}\times [-1, 1], $$ 
 where $\iota$  is the previous involution on $\Sigma_{2g+j-1}$. 
In the case $j=1$, 
$u$ fixes $2 = 2j$ curves 
$B \times \{1\} (= B \times \{-1\})$ and $B \times \{0\}$. 
In the case $j=2$, 
$u$ fixes $4= 2j$ curves $B \times \{1\}$, $B \times \{0\}$ and $B' \times \{1\}$, $B' \times \{0\}$.

The quotient of $T_{\widehat{h}_{g,j}}$ by the action of $u$ is the mapping torus $T_{h_{g, j}}$ of $h_{g, j}$ under the identification
$(y,t)\sim (y, -t)$ for all $(y,t) \in \partial T_{h_{g,j}}= \partial F_{g, j}\times [-1, 1]/_{(y,1)\sim (y, -1)}$. 
This is equivalent to identifying $\{y\} \times [-1, 0]$ with  $\{y\} \times [0, 1]$ for all  $y  \in \partial F_{g, j}$. 
The resulting quotient is homeomorphic to the manifold obtained from $T_{h_{g, j}}$ 
by collapsing the set $\{y\}\times S^1$ to a point for all $y\in \partial F_{g, j}$.  
Thus the quotient of $ T_{\widehat{h}_{g,j}}$ under $u$ is the relative mapping torus of $h_{g,j}$  which is
the open book decomposition $(F_{g, j}, h_{g,j})$ of $M$. 
In other words, 
the mapping torus 
$T_{\widehat{h}_{g, j}}$ of $\widehat{h}_{g, j}$ is a $2$-fold branched cover of $M$  branched cover the $2j$-component link that comes from 
 the above $2j$ curves fixed by $u$.   

This completes the proof of Claim 3. 
\medskip

By Claims 1 and 2, 
one can apply Theorem \ref{hyperbolic} to 
 the orbits of $\widehat{\gamma}_1, \dots, \widehat{\gamma}_k$ under $\widehat{h}_{g,j}$. 
 Then we have  $n \in {\Bbb N}$ given by Theorem \ref{hyperbolic}. 
 For ${\bm n}=(n_1, \dots, n_k)\in {\Z}^k$, 
 we set 
 $$f_{{\bm n}}= \tau_{\widehat{\gamma}_k}^{n_k} \dots \tau_{\widehat{\gamma}_1}^{n_1} \widehat{h}_{g,j} \hspace{5mm} \mbox{and} \hspace{5mm} 
 N_{{\bm n}} = T_{f_{{\bm n}}}.$$

By Theorem \ref{hyperbolic}(a), 
if $|n_i| \ge n$ for $i= 1, \dots, k$, 
then $f_{{\bm n}}$ is pseudo-Anosov and $N_{{\bm n}}$ is a hyperbolic $3$-manifold 
which is a $\Sigma_{2g+ j-1}$-bundle over $S^1$. 
Thus $N_{{\bm n}}$ has a property of Theorem  \ref{adding_arc}(a). 
By Theorem \ref{hyperbolic}(b),  $N_{{\bm n}}$ also has a property of Theorem  \ref{adding_arc}(c).

\begin{center}
\begin{figure}[t]
\includegraphics[height=10cm]{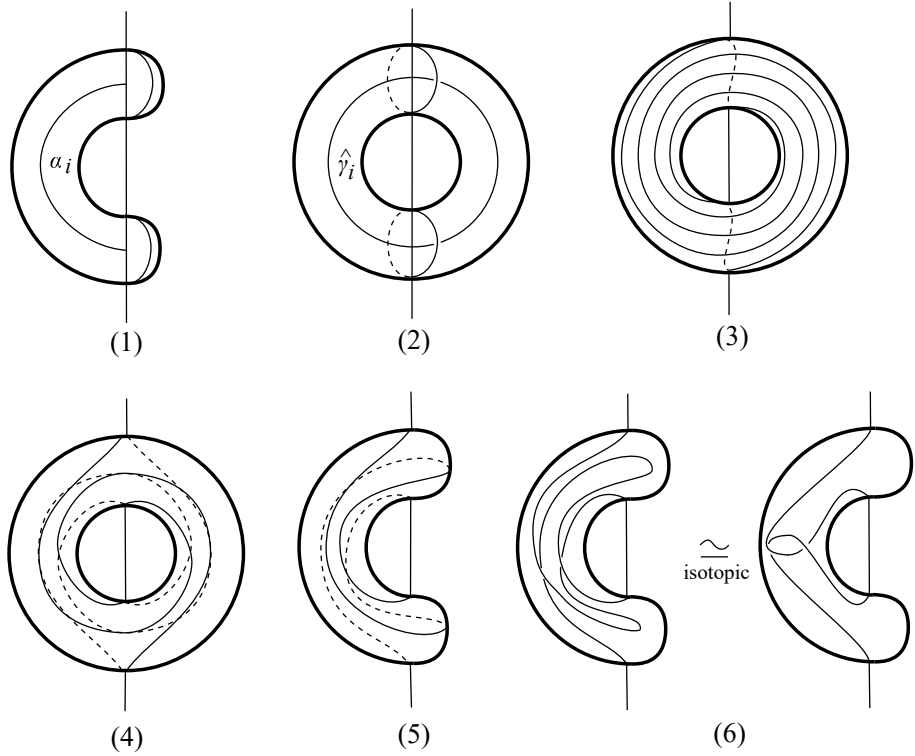}
\caption{(1) The regular neighborhood of $\alpha_i \times \{0\}$ in $M$, which is a 3-ball. 
(2) The regular neighborhood of $\widehat{\gamma}_i \times \{0\}$ in $T_{\widehat{h}_{g,j}}$,
which is a solid torus. 
(3)  In order to obtain $T_{\tau_{\widehat{\gamma}_i}^{2}\widehat{h}_{g,j}}$, we remove 
 the solid torus from $T_{\widehat{h}_{g,j}}$ and glue
this solid torus again so that the boundary of its meridian is identified with
one of the circles on the torus. 
(4) Isotope the circles to be invariant by the elliptic involution. 
(5) Take a quotient by the elliptic involution. 
(6) Push the arcs into the 3-ball, then we have the branched set of the new $2$-fold branched cover of $M$ after the Dehn surgery.} 

\label{fig_figure-claim4}
\end{figure}
\end{center}

\noindent
{\bf Claim 4.} 
For $j \in \{1,2\}$ and ${\bm n}=(n_1,\dots, n_k) \in {\Bbb  Z}^k$, 
the mapping torus $N_{\bm n}$ of $f_{{\bm n}}= \tau_{\widehat{\gamma}_k}^{n_k} \dots \tau_{\widehat{\gamma}_1}^{n_1} \widehat{h}_{g,j}$ 
 is a 2-fold branched cover of $M$ branched over a $2j$-component link.
\medskip

\noindent 
Proof of Claim 4.  
We use Montesinos' trick \cite{Mon2,Mon1}. 
(cf.  Auckly \cite[Example 1]{Auckly}.) 
See Figure \ref{fig_figure-claim4}, which illustrates Montesinos' trick. 
Since $\widehat{\gamma}_1, \dots, \widehat{\gamma}_k$ are mutually disjoint simple closed curves in $\Sigma_ {2g+ j-1}$, 
the curves 
$\delta_i^* = \widehat{\gamma}_i \times \{0\} \subset \Sigma_{2g+ j-1} \times \{0\}$  
for $i= 1, \dots, k$ are mutually disjoint. 
We consider the link 
$L^*_k = \delta_1^* \cup \dots \cup \delta_k^*$ in $T_{\widehat{h}_{g,j}}$. 
Then 
the $3$-manifold $N_{\bm n}$ is obtained from the mapping torus $T_{\widehat{h}_{g,j}}$ of $\widehat{h}_{g,j}$ 
by the Dehn surgery along the link $L^*_k $.

Notice that 
each $\delta_i^*$  is invariant under the involution 
$u:T_{\widehat{h}_{g,j}}\rightarrow  T_{\widehat{h}_{g,j}}$.   
Now  we do Dehn surgery along  $L^{*}_k$. 
For $i=1, \dots, k$, we remove the interior of a  neighborhood $N_i$ of $\delta_{i}^{*}$, and 
replace it with  a new solid torus $V_i$. 
The involution $u_i:= u|_{\partial N_i}: \partial N_i  \rightarrow \partial N_i$ 
extends to an  elliptic involution on the solid torus $V_i$ added with Dehn filling.  
The effect of the Dehn surgery on the quotient by the elliptic involution on $V_i$ 
 is a modification of the $3$-manifold $M$ inside a collection of 3-balls that changes the branched set 
 for the $2$-fold branched cover $T_{\widehat{h}_{g,j}} \rightarrow M$,  
but not the ambient manifold $M$. 
Thus $N_{\bm n}$ is still a $2$-fold branched covers of $M$ 
branched over a link with $2j$ components. 
This completes the proof of Claim 3. 
\medskip

By Claim 4, the manifold $N_{{\bm n}}$ satisfies a property of Theorem  \ref{adding_arc}(b), and 
we have finished the proof of Theorem \ref{adding_arc}. 
\end{proof}

\section{Large volume vs. fixed genus}
In this section we discuss conditions on 3-manifolds under which Question \ref{question} has a positive answer. 

\begin{thm}\label{criterion}
Let $M$ be a closed, connected, oriented $3$-manifold
containing a hyperbolic fibered knot of genus $g_0 \ge 2$. 
Suppose that for any $g\geq g_0$ and $j \in \{1,2\}$, 
$M$ contains a family $\{ K^{j}_r(g)\}_{r \in {\Bbb N}}$ of hyperbolic fibered links of genus $g$ with $j$ components 
such that $\mathrm{vol}(M\setminus K^{j}_r(g))\to \infty$ as $r\to \infty.$ 
Then Question \ref{question} has a positive answer for $M$.
\end{thm}

\begin{proof} For notational simplicity we will assume that $j=1$. 
The case $j=2$ is analogous.
Fix  $g\geq g_0$, we consider the family of the hyperbolic knots 
$\{K_r:= K^{1}_r(g)\}_{r \in {\Bbb N}}$ satisfying the assumption of Theorem \ref{criterion},  
where  $F_ {r}= \Sigma_ {g,1}$
denotes the fiber of $K_r$ and $h_r:= h_{g,1}$ the monodromy.
As in the proof of Theorem \ref{adding_arc}, for any $r \in {\Bbb N}$, 
we build the 2-fold branched cover
$T_{\widehat{h}_{r}}$  of  $M$
with fiber
$\Sigma_{2g}=DF_{r}$ and monordomy
$ \widehat{h}_{r} = h_{r}\# h_{r}^{-1}$.
 We apply the process in the proof of Claims 1--4 of the proof of Theorem \ref{adding_arc}: 
 For  $1<k \leq g$, say for $k=2$,
 we take simple closed curves
 $\widehat{\gamma^r}_1,\widehat{\gamma^r}_2$ on $\Sigma_ {2g}$ 
 so that Theorem \ref{hyperbolic} can be applied. 
 By pushing  $\widehat{\gamma^r}_1$, $\widehat{\gamma^r}_2$, 
 we get a hyperbolic link $L^r$  in $T_{\widehat{h}_{r}}$.  Recall that for  ${\bm n}=(n_1,n_2) \in {\Bbb  Z}^2$ with $n_1,n_2$ large enough, 
 the mapping class 
 $f^r_{{\bm n}}:= \tau_{\widehat{\gamma^r}_2}^{n_2}  \tau_{\widehat{\gamma^r}_1}^{n_1} \widehat{h}_{r} $ 
 defined on $\Sigma_ {2g}$ 
 is pseudo-Anosov and its mapping torus  $N^r_{\bm n}$ is hyperbolic. 
 Furthermore, $N^r_{\bm n}$ is obtained by Dehn filling of 
 $\overline{T_{\widehat{h}_{r}}\setminus  \mathcal{N}(L^r)}$, 
 and we have
\begin{equation} 
\label{eq:variation3}
 \frac{1}{2}\ \vol(T_{\widehat{h}_{r}}\setminus L^r)< \vol(N^r_{\bm n}),
\end{equation}
where \eqref{eq:variation3} follows from \eqref{eq:variation2}.
To finish the proof of the theorem, we need to show that $\vol(N^r_{\bm n})\to \infty$ as $r\to \infty$.
This follows immediately from \eqref{eq:variation3} and the following. 
 \medskip
 
\noindent
{\bf Claim 1.}  
We have $\vol(T_{\widehat{h}_{r}}\setminus L^r)\to \infty$ as $r\to \infty$.
 \medskip

\noindent
Proof of Claim 1. For any orientable 3-manifold
$X$ with $\partial X$ empty or $\partial X$ consisting only of tori, 
let $||X||$ denote  the Gromov norm of $X$. See 
\cite[Definition 6.1.2, the beginning of Section 6.5]{thurston:notes}.
If $X$ is closed and hyperbolic, or if $\partial X$ consists only of tori and the interior of $X$ is hyperbolic, then
$v_3  ||X||= \vol(X)$, where $v_3$ is the volume of the ideal regular tetrahedron. 
 (See \cite[Theorem 6.2, Lemma 6.5.4]{thurston:notes}.) 
By construction, $M\setminus K_r \simeq M \setminus \mathcal{N}(K_r)$  
is a submanifold of $T_{\widehat{h}_{r}}$ and  $\partial(M\setminus K_r)$ is an incompressible torus in $T_{\widehat{h}_{r}}$. Indeed, we can think of $T_{\widehat{h}_{r}}$ as obtained by identifying two copies of $M\setminus K_r$ along their torus boundary.
By \cite[Theorem 6.5.5]{thurston:notes}, 
we have
$$v_3 \ ||T_{\widehat{h}_{r}}||\geq  \mathrm{vol}(M\setminus K_r),$$
which implies that 
$||T_{\widehat{h}_{r}}||\to \infty$ as $r\to \infty.$  
Since $T_{\widehat{h}_{r}}$ is obtained from $T_{\widehat{h}_{r}} \setminus L^r$ by adding solid tori, by 
\cite[Proposition 6.5.2]{thurston:notes} we obtain
$$ \vol(T_{\widehat{h}_{r}}\setminus L^r)=v_3\  ||T_{\widehat{h}_{r}}\setminus L^r ||\geq v_3\   ||T_{\widehat{h}_{r}}||,$$
and $ \vol(T_{\widehat{h}_{r}}\setminus L^r) \to \infty$ as $r\to \infty.$
\end{proof}

For $M=S^3$, families of knots satisfying the assumption of Theorem \ref{criterion} are constructed 
in Futer-Purcell-Schleimer \cite[Theorem 1]{FPS22}. 
Thus we have the following result. 

\begin{cor} \label{positive} 
For any $g\geq 2$, 
the set  $\mathcal{D}_{2g}(S^3)$ contains an infinite family of pseudo-Anosov elements whose mapping tori have arbitrarily large volume.
\end{cor}

\subsection*{Acknowledgement} This work started from  a discussion of the authors during the conference 
``Classical and quantum $3$-manifold topology" held at Monash University (Melbourne, Australia)
in December of 2018. We thank the organizers (D. Futer, S. Garoufalidis, C. Hodgson, J. Purcell,  H. Rubinstein, S. Schleimer, P. Wedrich) for inviting us to participate and for ensuring  excellent working conditions during the conference. 
We thank M. Sakuma for explaining his work \cite{KS22} and for pointing out a relation between our results and \cite[Question 9.7]{KS22}. 
We thank Y. Koda for helpful comments. 
We thank D. Futer for explaining his work \cite{FPS22} with J. S. Purcell and S. Schleimer.

\bibliographystyle{hamsplain}
\bibliography{biblio}

\def\cprime{$'$}
\providecommand{\bysame}{\leavevmode\hbox to3em{\hrulefill}\thinspace}
\providecommand{\href}[2]{#2}
\begin{thebibliography}{10}

\bibitem{Adams}
C.~Adams, \emph{Volumes of {$N$}-cusped hyperbolic {$3$}-manifolds}, J. London
  Math. Soc. (2) \textbf{38} (1988), no.~3, 555--565.

\bibitem{Auckly}
D.~Auckly, \emph{Two-fold branched covers}, J. Knot Theory Ramifications
  \textbf{23} (2014), no.~3, 1430001, 29.

\bibitem{brooks}
R.~Brooks, \emph{On branched coverings of {$3$}-manifolds which fiber over the
  circle}, J. Reine Angew. Math. \textbf{362} (1985), 87--101.

\bibitem{ColinHonda}
V.~Colin and K.~Honda, \emph{Stabilizing the monodromy of an open book
  decomposition}, Geom. Dedicata \textbf{132} (2008), 95--103.

\bibitem{DK:cosets}
R.~Detcherry and E.~Kalfagianni, \emph{Cosets of monodromies and quantum
  representations}, \emph{Indiana Univ. Mathematics Journal} {\bf 71} (2022),
  no. 3, 1101--1129.

\bibitem{Fathi}
A.~Fathi, \emph{Dehn twists and pseudo-{A}nosov diffeomorphisms}, Invent. Math.
  \textbf{87} (1987), no.~1, 129--151.

\bibitem{fkp:filling}
D.~Futer, E.~Kalfagianni, and J.S. Purcell, \emph{{Dehn filling, volume, and
  the Jones polynomial}}, J. Differential Geom. \textbf{78} (2008), no.~3,
  429--464.

\bibitem{FPS22}
D.~Futer, J.~S. Purcell, and S.~Schleimer, \emph{Large volume fibred knots of
  fixed genus}, Preprint, 2022.

\bibitem{Hatcher}
A.~Hatcher, \emph{Notes on basic 3-manifold topology}, {\tt
  http://www.math.cornell.edu/\allowbreak \~{ }hatcher/3M/\allowbreak
  3Mdownloads.html}.

\bibitem{HiroseKin20}
S.~Hirose and E.~Kin, \emph{On hyperbolic surface bundles over the circle as
  branched double covers of the 3-sphere}, Proc. Amer. Math. Soc. \textbf{148}
  (2020), no.~4, 1805--1814.

\bibitem{KS22}
Y.~Koda and M.~Sakuma, \emph{Homotopy motions of surfaces in 3--manifolds},
  \emph{Quart. J. Math.} (2022), 1--43.

\bibitem{LongMorton}
D.~D. Long and H.~R. Morton, \emph{{Hyperbolic 3-manifolds and surface
  automorphisms}}, {Topology} \textbf{25} ({1986}), no.~4, {575--583}.

\bibitem{MaMi}
H.~A. Masur and Y.~N. Minsky, \emph{Geometry of the complex of curves. {I}.
  {H}yperbolicity}, Invent. Math. \textbf{138} (1999), no.~1, 103--149.

\bibitem{Mon2}
J.~M. Montesinos, \emph{Surgery on links and double branched covers of
  {$S^{3}$}}, Knots, groups, and {$3$}-manifolds ({P}apers dedicated to the
  memory of {R}. {H}. {F}ox), 1975, pp.~227--259. Ann. of Math. Studies, No.
  84.

\bibitem{Mon1}
\bysame, \emph{On {$3$}-manifolds having surface bundles as branched
  coverings}, Proc. Amer. Math. Soc. \textbf{101} (1987), no.~3, 555--558.

\bibitem{Purcell}
J.~S. Purcell, \emph{Hyperbolic knot theory}, Graduate Studies in Mathematics,
  vol. 209, American Mathematical Society, Providence, RI, [2020] \copyright
  2020.

\bibitem{sakuma}
M.~Sakuma, \emph{Surface bundles over {$S^{1}$} which are {$2$}-fold branched
  cyclic coverings of {$S^{3}$}}, Math. Sem. Notes Kobe Univ. \textbf{9}
  (1981), no.~1, 159--180.

\bibitem{Soma84}
T.~Soma, \emph{Hyperbolic, fibred links and fibre-concordances}, Math. Proc.
  Cambridge Philos. Soc. \textbf{96} (1984), no.~2, 283--294.

\bibitem{thurston:notes}
W.~P. Thurston, \emph{The geometry and topology of three-manifolds}, Princeton
  Univ. Math. Dept. Notes, 1979.

\bibitem{thurston:survey}
\bysame, \emph{Three-dimensional manifolds, {K}leinian groups and hyperbolic
  geometry}, Bull. Amer. Math. Soc. (N.S.) \textbf{6} (1982), no.~3, 357--381.

\bibitem{thurston:mappingtori}
\bysame, \emph{{On the geometry and dynamics of diffeomorphisms of surfaces.}},
  {Bull. Amer. Math. Soc.} \textbf{19} ({1988}), no.~2, {417--431}.

\bibitem{Thurston98}
\bysame, \emph{Hyperbolic structures on $3$-manifolds ii: Surface groups and
  $3$-manifolds which fiber over the circle}, arXiv preprint arXiv.math/9801045
  (1998).

\end{thebibliography}

\end{document}